\documentclass[11pt]{article}
\usepackage[utf8]{inputenc}
\usepackage{amsmath, amssymb, amsthm}
\usepackage{geometry}
\usepackage{hyperref}
\usepackage{enumitem}
\usepackage{graphicx}
\usepackage{mathrsfs}
\usepackage{bm}
\usepackage{tikz}
\usepackage{mathtools}
\usepackage{dsfont}
\usepackage{pgf,pgfplots,tikz} 
\usepackage{float}
\usepackage{booktabs}

\newtheorem{Theorem}{Theorem}[section]
\newtheorem{Lemma}[Theorem]{Lemma}
\newtheorem{Definition}[Theorem]{Definition}
\newtheorem{Proposition}[Theorem]{Proposition}
\newtheorem{Corollary}[Theorem]{Corollary}

\theoremstyle{remark}
\newtheorem{Remark}[Theorem]{Remark}
\newtheorem{Example}[Theorem]{Example}

\newcommand{\cut}{\operatorname{CUT}}

\def\N{{\mathbb N}}

\newcommand{\cutn}{\ensuremath{\operatorname{CUT}(n)}}

\newcommand{\vn}{v^n}

\newcommand{\ind}{\mathds{1}}

\newcommand{\bone}{{\mathbf{1}}}
\newcommand{\bx}{{\mathbf{x}}}

\newcommand{\by}{{\mathbf{y}}}

\newcommand{\lambdahash}{\lambda^{\#}}
\newcommand{\lambdas}{\lambda^s}

\newcommand{\code}{\operatorname{code}}
\newcommand{\decode}{\mathrm{decode}}

\title{An Explicit Formula for Vertex Enumeration in the CUT(n) Polytope via Probabilistic Methods}
\author{Nevena Mari\'c\thanks{School of Computing, Union University, Belgrade, Serbia. email: nmaric@raf.rs. \\
An extended abstract related to this work appeared in \cite{maric2018cut}}}
\date{\today}

\begin{document}

\maketitle
%
%
\begin{abstract}
We present an explicit closed-form formula for the vertices of the classical cut polytope $\operatorname{CUT}(n)$, defined as the convex hull of cut vectors of the complete graph $K_n$. Our derivation proceeds via a related polytope, denoted $\mathbf{1}$-$\operatorname{CUT}(n)$, whose vertices are obtained by flipping all bits of the $\operatorname{CUT}(n)$ vertices. This polytope arises naturally in a probabilistic context involving agreement probabilities among symmetric Bernoulli random variables which serves as the starting point of this work.

Our approach constructs the vertex set recursively via a binary encoding that stems from this probabilistic perspective. We prove that the resulting sequence of encoded integers, when appropriately scaled, exhibits an almost-linear behavior closely approximating the line $y = x - \frac{1}{2}$. This structure motivates the introduction of the alternating cycle function, an integer-valued map whose key property is power-of-two composition invariance. The function serves as the foundation for our closed-form enumeration formula.

The result provides a rare instance of explicit vertex characterization for a $0$/$1$-polytope and offers a transparent combinatorial construction independent of enumeration algorithms.
\end{abstract}



\section{Introduction}\label{intro}

The \emph{vertex enumeration problem} - listing all vertices of a polyhedron defined by linear inequalities - is a foundational challenge in combinatorial optimization and computational geometry~\cite{matheiss1980survey,dyer1983complexity}. This problem is central to a wide range of applications, from operations research to computational biology, and has motivated the development of several algorithmic paradigms (see, e.g. \cite{dyer1977algorithm} for an early review of applications). 
For bounded polyhedra, vertex enumeration problem is equivalent to facet generation - that is, listing all facets of a polytope when its vertices are explicitly given~\cite{lovasz1992combinatorial}.

A variety of algorithmic approaches to vertex and facet enumeration have been proposed in the literature (see e.g. \cite{elbassioni2018enumerating} and references therein). 
There is a large body of literature on vertex enumeration, and nearly all of it focuses on search-based algorithms for generating vertices or facets. This paper takes a different perspective: rather than designing new enumeration procedures, we provide an explicit closed-form formula for all vertices of a particular polytope. While our approach is not algorithmic in nature, we still mention key contributions in this area for context.

Classical algorithmic approaches include the double description method \cite{motzkin1953double}, which incrementally builds the vertex set, and pivoting or reverse search algorithms \cite{avis1991pivoting}, \cite{avis2000revised}, which traverse the polytope's skeleton. Despite substantial progress, the complexity of vertex enumeration remains prohibitive for general polyhedra, with NP-hardness results for unbounded cases and only partial tractability for special classes~\cite{dyer1983complexity}.  For more on computational complexity of the problem see ~\cite{khachiyan2008generating} and references therein.

A particularly important family of polytopes arises as the convex hulls of binary vectors, known as \emph{0/1-polytopes}. These structures encode a wide variety of combinatorial objects, such as matchings, stable sets, and cut vectors, and their binary nature enables specialized algorithmic techniques. Bussieck and L{\"u}bbecke~\cite{bussieck1998vertex} demonstrated that vertex enumeration for 0/1-polytopes can be reduced to a sequence of linear programs, while Behle et al.~\cite{behle20070} introduced binary decision diagrams (BDDs) to efficiently encode and enumerate vertices, exploiting the underlying binary structure. More recently, Merino and M{\"u}tze~\cite{merino2024traversing} improved enumeration delay bounds by leveraging Gray codes and Hamilton paths on 0/1-polytope skeletons. Binary encoding not only facilitates algorithmic efficiency but also reveals  combinatorial and algebraic properties~\cite{codes2008}.

A central example in this context is the \emph{cut polytope} $\mathrm{CUT}(n)$, introduced by Barahona and Mahjoub~\cite{barahona1986cut}, which is defined as the convex hull of cut vectors of the complete graph $K_n$. Each vertex of $\mathrm{CUT}(n)$ corresponds to a bipartition $(S, \bar{S})$ of the vertex set $\{1, 2, \ldots, n\}$, encoded as a binary vector in $\{0,1\}^{\binom{n}{2}}$ where the $(i,j)$-th coordinate is 1 if vertices $i$ and $j$ belong to different parts of the bipartition, and 0 otherwise. The cut polytope plays a pivotal role in combinatorial optimization, most notably as the feasible region for the max-cut problem, which has applications ranging from circuit design to statistical physics, particularly in the study of spin glasses~\cite{dezalaurent1997}. An excellent overview of 0/1-polytopes and their applications, including the cut polytope, is given in~\cite{ziegler2000}.

The cut polytope has gained additional significance through a probabilistic interpretation developed in Huber and Mari\'c~\cite{huberm2017}. They addressed the basic question of which correlation patterns can arise among symmetric Bernoulli random variables (binary variables taking values $0$ or $1$ with equal probability) by analyzing \emph{agreement probabilities} $\mathbb{P}(B_i = B_j)$ for pairs of such variables. Their main result showed that the set of all attainable agreement probabilities forms a convex polytope whose vertices are obtained by flipping all coordinates of the $\cutn$ vertices. We refer to this bitwise complement of the cut polytope as $\bone$-$\cutn$. It should be noted that some probabilistic interpretation of the cut polytope can be dated back to the work of Avis~\cite{avis1977}.

Despite the cut polytope's  importance, explicit formulas for its vertices have remained elusive. Existing methods for vertex enumeration, such as those based on BDDs or iterative traversal, provide algorithmic solutions but do not yield direct combinatorial descriptions of the vertex structure. 

In this work, we address this gap by deriving an explicit closed-form expression for the vertex set of $\operatorname{\bone\text{-}CUT}(n)$ polytope. Our approach generates the vertices directly via a periodic integer - generating  function, bypassing iterative enumeration entirely.

\medskip

\noindent
\textbf{Main results.} \
Our starting point is the probabilistic interpretation of the $\bone$-$\operatorname{CUT}(n)$ polytope in terms of agreement probabilities among symmetric Bernoulli random variables. From this perspective, we construct an explicit algorithm that generates the full set of vertices without resorting to iterative enumeration or traversal of the polytope's structure.

We then introduce a binary encoding that maps each vertex to a unique integer. This encoding allows us to view the vertex set as integers. When plotted, and after appropriate scaling, these integers reveal an almost-linear alignment. This phenomenon is not merely visual: we formalize it in a proposition that shows the scaled vertex sequence lies close to the line $y = x - \frac{1}{2}$ and quantifies the deviation from linearity.

The recurrence relations observed among the encoded vertices lead to the definition of a new integer-valued function, the \emph{alternating cycle function}. While reminiscent of discrete triangular waves, this function exhibits distinct compositional properties. It serves as the foundation of our explicit closed-form formula for the vertices of $\bone$-$\operatorname{CUT}(n)$ (Theorem~\ref{Thm:1}).

This result provides a direct, closed-form description of the entire vertex set and reveals structural regularities not evident from algorithmic enumeration alone.

\medskip

\noindent
The remainder of the paper is organized as follows. Section~2 reviews the definition and basic properties of cut polytopes. Section~3 discusses the probabilistic interpretation of $\bone-\operatorname{CUT}(n)$ in terms of agreement probabilities. Section~4 introduces the binary encoding framework. Section~5 develops the recursive formula for vertex enumeration. Section~6 explores symmetries and compositional properties of the alternating cycle function, and Section~7 presents the main theorem with a worked example and a corollary for $\cutn$. Concluding remarks are provided in Section~8. An appendix contains an algorithm for computing the vertices.

\section{Cut polytopes}
We begin by recalling the classical definitions and key properties of cut polytopes, which constitute the central focus of this work.

Let $G = (V, E)$ be a graph with vertex set $V$ and edge set $E$. For $S \subseteq V$, a \emph{cut} is the partition $(S, S^C)$ of the vertices, and the \emph{cut-set} consists of all edges connecting a node in $S$ to a node in $S^C$.

Let $V_n = [n] = \{1, \ldots, n\}$, $E_n = \{(i, j) : 1 \leq i < j \leq n\}$, and $K_n = (V_n, E_n)$ be the complete graph on $n$ vertices.

\begin{Definition}[Cut vector]
For every $S \subseteq [n]$, the \emph{cut vector} $\delta(S) \in \{0,1\}^{E_n}$ is defined by
\[
\delta(S)_{ij} = 
\begin{cases}
1, & \text{if } |S \cap \{i, j\}| = 1, \\
0, & \text{otherwise}
\end{cases}
\]
for all $1 \leq i < j \leq n$.
\end{Definition}

\begin{Definition}[Cut polytope]
The \emph{cut polytope} $\operatorname{CUT}(n)$ is the convex hull of all cut vectors of $K_n$:
\[
\operatorname{CUT}(n) = \operatorname{conv}\left\{ \delta(S) : S \subseteq [n] \right\}.
\]
\end{Definition}

\begin{Remark}
Since $\delta(S) = \delta([n] \setminus S)$ (complementary subsets yield identical cut vectors), the polytope $\operatorname{CUT}(n)$ has exactly $2^{n-1}$ distinct vertices, corresponding to the $2^{n-1}$ non-equivalent bipartitions of $[n]$. Each $\delta(S)$ is a $0/1$-vector in $\{0,1\}^{\binom{n}{2}}$, so $\operatorname{CUT}(n)$ is a $0/1$-polytope of dimension $\binom{n}{2}$. For a comprehensive treatment, see~\cite{ziegler2000,dezalaurent1997}.
\end{Remark}
\begin{Remark}\label{remark:n12}
For small values of \( n \), the vertices of \( \operatorname{CUT}(n) \) are easy to describe. When \( n = 1 \), the complete graph \( K_1 \) has no edges, so \( \operatorname{CUT}(1) \) consists of a single point in \( \mathbb{R}^0 \). For \( n = 2 \), the graph \( K_2 \) has one edge, and the cut vectors are $\{0, 1\}$, forming the vertices of \( \operatorname{CUT}(2) \) in \( \mathbb{R}^1 \). 
\end{Remark}
\begin{Example}
For $n=3$, the cut polytope $\operatorname{CUT}(3)$ has $2^{2} = 4$ vertices:
\[
\delta(\{1\}) = (1,1,0), \quad \delta(\{2\}) = (1,0,1), \quad \delta(\{3\}) = (0,1,1), \quad \delta(\emptyset) = (0,0,0).
\]
These form the vertices of a tetrahedron in $\mathbb{R}^3$.
\end{Example}

\noindent
The cut polytope is a well-studied $0/1$-polytope with deep connections to combinatorial optimization, most notably as the feasible region for the max-cut problem. Its structure and properties are also central in the study of correlations among symmetric Bernoulli random variables, as will be discussed in Section~3.

\section{Probabilistic interpretation of cut polytopes}

It is well known that the correlation between two random variables is constrained to the interval $[-1,1]$. However, for specific distributions, the attainable range of correlations is often strictly smaller and depends on the marginal distributions. These extremal values, known as Fr{\'e}chet bounds, are theoretically known \cite{whitt1976} but are rarely explicit for general distributions. As the dimension increases, the problem of characterizing all attainable correlation matrices becomes considerably more complex. For $n \geq 3$, a complete description is generally unavailable, except in special cases.

A correlation matrix is a symmetric positive semi-definite matrix with all diagonal entries equal to $1$. The set of all such matrices of order $n$ is denoted by $\mathcal{E}_n$ and is sometimes called the \emph{elliptope}~\cite{laurent1995}. For Gaussian marginals, every point in $\mathcal{E}_n$ is attainable, but for other distributions, very little is known. 

The case of symmetric Bernoulli distributions represents a notable exception where complete characterization is possible~\cite{huberm2017}. This tractability arises from the discrete nature of the distribution and the specific symmetry properties of binary random variables. Moreover, this special case has proven useful as a key component in algorithmic strategies for analyzing the attainable correlation structures of more general distributions~\cite{huberm2015}.

A \emph{symmetric Bernoulli random variable} is a binary random variable taking values $0$ and $1$ with equal probability. This distribution will be referred to as $\operatorname{Bern}(1/2)$. Let $\mathcal{B}_n$ denote the set of all $n$-variate symmetric Bernoulli distributions, i.e., probability measures $\mu$ on $\{0,1\}^n$ such that each marginal is $\operatorname{Bern}(1/2)$:
\[
\sum_{\mathbf{x} \in \{0,1\}^n : \mathbf{x}(k) = 0} \mu(\mathbf{x}) = \frac{1}{2} \quad \text{for } k = 1, \ldots, n.
\]

\begin{Remark}
The constraint that each marginal distribution is $\operatorname{Bern}(1/2)$ defines a system of linear equalities on the probability simplex, ensuring that $\mathcal{B}_n$ is a polytope~\cite{devroye2015copulas}.
\end{Remark}

For $\mu \in \mathcal{B}_n$, let $\mathbb{P}_\mu(\cdot)$, $\mathbb{E}_\mu[\cdot]$, $\operatorname{Cov}_\mu(\cdot, \cdot)$, and $\operatorname{Corr}_\mu(\cdot, \cdot)$ denote probability, expectation, covariance, and correlation under $\mu$.

\begin{Definition}[Agreement probability]
Let $(B_1, \ldots, B_n) \sim \mu$, where $\mu \in \mathcal{B}_n$. For all $1 \leq i < j \leq n$, the \emph{agreement probability} is defined as
\[
\lambda_\mu(i, j) := \mathbb{P}_\mu(B_i = B_j).
\]
We collect these values into the vector
\[
\lambda_\mu := \left( \lambda_\mu(i, j) \right)_{1 \leq i < j \leq n} \in [0,1]^{\binom{n}{2}}.
\]
The map sending $\mu$ to its vector of agreement probabilities is denoted by
\[
\Lambda : \mathcal{B}_n \to [0,1]^{\binom{n}{2}}, \qquad \Lambda(\mu) = \lambda_\mu.
\]
\end{Definition}

\begin{Remark}
Although the agreement probability $\lambda_\mu(i, j)$ can be defined for all pairs $i, j$, we include in the vector $\lambda_\mu$ only the entries with $i < j$. This avoids redundancy, since $\lambda_\mu(i, j) = \lambda_\mu(j, i)$ by symmetry, and ensures that each unordered pair is represented exactly once. This convention is standard in the study of the cut polytope, where each coordinate of the vector corresponds uniquely to an edge of the complete graph $K_n$, indexed by pairs with $1 \leq i < j \leq n$. The precise relationship between the agreement probability map $\Lambda$ and the cut polytope $\operatorname{CUT}(n)$ will be established in Section~3.1.
\end{Remark}

The pairwise correlation between $B_i$ and $B_j$ under $\mu$ is given by
\[
 \operatorname{Corr}_\mu(B_i, B_j) = 4 \mathbb{E}_\mu[B_i B_j] - 1.
\]

\begin{Remark}
The agreement probability and correlation are linearly related:
\[
\lambda_\mu(i, j) = \mathbb{P}_\mu(B_i = B_j) = \frac{1}{2}\left(1 + \operatorname{Corr}_\mu(B_i, B_j)\right).
\]
\end{Remark}

Huber and Mari\'c~\cite{huberm2017} studied elements of $\mathcal{B}_n$ via their agreement probabilities and provided necessary and sufficient conditions for an agreement matrix to be attainable for general $n$. They showed that the set of all attainable agreement probability vectors forms a convex polytope, specifically the $\bone$-$\operatorname{CUT}(n)$ polytope, as described below.

\subsection{Cut polytopes via agreement probabilities}

To establish the connection between agreement probabilities and cut polytopes, we define diagonal distributions in $\mathcal{B}_n$. Let $\mathbf{1}$ denote the all-ones vector in $\mathbb{R}^n$. For each $\mathbf{x} \in \{0,1\}^n$, define the probability measure $\pi_{\mathbf{x}}$ on $\{0,1\}^n$ by
\[
\pi_{\mathbf{x}}(\mathbf{x}) = \pi_{\mathbf{x}}(\mathbf{1} - \mathbf{x}) = \frac{1}{2}, \qquad \pi_{\mathbf{x}}(\mathbf{y}) = 0 \text{ for } \mathbf{y} \notin \{\mathbf{x}, \mathbf{1} - \mathbf{x}\}.
\]

\begin{Remark}
Note that $\pi_{\mathbf{x}} \in \mathcal{B}_n$ and $\pi_{\mathbf{x}} = \pi_{\mathbf{1} - \mathbf{x}}$. These distributions are important because their agreement probability vectors are precisely the extreme points of $\Lambda(\mathcal{B}_n)$.
\end{Remark}

For notational convenience, we write
\[
\lambda(\mathbf{x}) := \lambda_{\pi_{\mathbf{x}}} = \Lambda(\pi_{\mathbf{x}}).
\]

The correspondence between these sets is characterized by the following:

\begin{Proposition}[{\cite[Thm.~4]{huber2017arxiv}, \cite[Thm.~1]{huberm2017}}]
\label{prop:lambda-cut-relation}
The agreement probability map $\Lambda$ and the cut polytope satisfy:
\begin{enumerate}
    \item $\Lambda(\mathcal{B}_n) = \mathbf{1} - \operatorname{CUT}(n)$.
    \item The vertices of $\Lambda(\mathcal{B}_n)$ are precisely the agreement probability vectors of diagonal symmetric Bernoulli distributions:
    \[
    \operatorname{vert}(\Lambda(\mathcal{B}_n)) = \{\lambda(\mathbf{x}) : \mathbf{x} \in \{0,1\}^n\}.
    \]
\end{enumerate}
\end{Proposition}

\begin{Remark} \label{remark:lambda}
Since $\lambda(\mathbf{x}) = \lambda(\mathbf{1} - \mathbf{x})$ for all $\mathbf{x} \in \{0,1\}^n$, each vertex of $\Lambda(\mathcal{B}_n)$ corresponds to two complementary binary vectors. To avoid redundancy, we consider only the \emph{canonical representatives} from each equivalence class. Following standard convention, we choose representatives with $x_1 = 1$ (the "upper half" of $\{0,1\}^n$), which reduces the vertex count to $2^{n-1}$, matching the known structure of $\operatorname{CUT}(n)$.
\end{Remark}

\begin{Lemma} \label{deflambda}
For any $\mathbf{x} \in \{0,1\}^n$, the agreement probability vector $\lambda(\mathbf{x})$ is given coordinate-wise by
\[
\lambda(\mathbf{x})_{ij} = \ind(x_i = x_j) \quad \text{for all } 1 \leq i < j \leq n,
\]
where $\ind(\cdot)$ denotes the indicator function
\[
\ind(A) =
\begin{cases}
1 & \text{if } A \text{ is true}, \\
0 & \text{otherwise}.
\end{cases}
\]
\end{Lemma}

\begin{proof}
This follows directly from the definition of $\pi_{\mathbf{x}}$ and the linearity of expectation. For any $1 \leq i < j \leq n$,
\[
\lambda(\mathbf{x})_{ij} = \mathbb{P}_{\pi_{\mathbf{x}}}(B_i = B_j) = \frac{1}{2}\ind(x_i = x_j) + \frac{1}{2}\ind((1-x_i) = (1-x_j)) = \ind(x_i = x_j).
\]
For additional details, see~\cite[Proposition~4.2]{huber2017arxiv}.
\end{proof}

This lemma highlights the direct combinatorial link between binary vectors and agreement probability vectors, setting the stage for our encoding-based approach to vertex enumeration.

Note that if $\mathbf{x}$ has $n$ coordinates, then $\lambda(\mathbf{x})$ has $\binom{n}{2}$ entries.

\begin{Example}
Applying Lemma~\ref{deflambda} to the following examples:
\begin{enumerate}
\item $\lambda(1,1,0) = (1,0,0)$, since $x_1 = x_2 = 1$, $x_1 \neq x_3$, and $x_2 \neq x_3$.
\item $\lambda(0,1,1,0) = (0,0,1,1,0,0)$, since only $x_2 = x_3 = 1$ and $x_1 = x_4 = 0$.
\item $\lambda(1,1,0,1,0) = (1,0,1,0,0,1,0,0,1,0)$, with agreements between positions $(1,2)$, $(1,4)$, $(3,5)$, and $(4,4)$ (but not $(4,4)$ since $i < j$).
\end{enumerate}
\end{Example}

\begin{Remark}
It is worth noting that if one considers disagreement probabilities, i.e., $\mathbb{P}(B_i \neq B_j)$, instead of agreement probabilities, the resulting attainable set forms the $\operatorname{CUT}(n)$ polytope directly, rather than $\bone$-$\operatorname{CUT}(n)$. In this work, we follow the approach from~\cite{huberm2017}, focusing on agreement probabilities because they are more intuitively related to pairwise correlations and provide a natural probabilistic interpretation that aligns with correlation analysis.
\end{Remark}

This probabilistic interpretation of the cut polytope motivates the binary encoding and explicit vertex enumeration developed in the following sections.

\section{Binary encoding}
Having established the connection between agreement probabilities and cut polytopes, we now turn to the question of how to efficiently enumerate and represent their vertices. To this end, we introduce a binary encoding scheme that will play a key role in our explicit formula.

In order to enumerate vertices of $\bone$-$\operatorname{CUT}(n)$, we use binary encoding of $0/1$-vectors. Instead of elements $\bx \in \{0,1\}^n$, we work with integers $\{0,1,2,\ldots,2^n-1\}$. The main question is: \emph{Where does the map $\lambda$ send these integers?} This will be the focus of the rest of the paper.


\begin{Definition}[Binary encoding]
Let $\bx = (x_1, x_2, \ldots, x_n) \in \{0,1\}^n$. The binary encoding function
\[
\operatorname{code} : \{0,1\}^n \to \mathbb{N}
\]
is defined by
\[
\operatorname{code}(\bx) = \sum_{j=1}^n x_j \cdot 2^{n-j}.
\]
\end{Definition}

\noindent To ensure that the binary decoding from integers to vectors is also well-defined and unique, we fix the length of the binary vector in the following definition.

\begin{Definition}[Binary decoding]
Let $n \geq 1$ and $k \in \{0, 1, \ldots, 2^n - 1\}$. The binary decoding function
\[
\decode_n : \{0, 1, \ldots, 2^n - 1\} \to \{0,1\}^n
\]
is defined as the unique binary vector $\bx = (x_1, \ldots, x_n) \in \{0,1\}^n$ such that
\[
k = \sum_{j=1}^n x_j \cdot 2^{n-j}.
\]
If $k < 2^n$, $\decode_n(k)$ pads $k$'s binary representation with leading zeros to ensure a length-$n$ vector.
\end{Definition}

For example, $\code(0,1,1,1)=7$ and $\decode_4(7)=(0,1,1,1)$, but $\decode_3(7)=(1,1,1)$ and $\decode_5(7)=(0,0,1,1,1)$.

Since we will work simultaneously with $0/1$ vectors, their $0/1$ string versions, and corresponding integers, we define three levels of the $\lambda$ map. Let $m = \binom{n}{2}$.

\begin{description}
\item[Vector-level map:] As used so far, $\lambda : \{0,1\}^n \to \{0,1\}^m$ is a function on binary vectors. For $\bx = (x_1, \ldots, x_n) \in \{0,1\}^n$, the output is another binary vector $\lambda(\bx) \in \{0,1\}^m$.
\item[Integer-level map:] Define the induced map $\lambda^\# : \{0,1,\ldots,2^n-1\} \to \{0,1,\ldots,2^m-1\}$ by
\[
\lambda^\#(k) := \operatorname{code}\left( \lambda\left( \decode_n(k)\right) \right),
\]
where $\decode_n(k) \in \{0,1\}^n$ is the binary vector representation of $k$, and $\operatorname{code}$ maps a vector back to its integer representation.
\item[String-level map:] For binary strings $\bx_s \in \{0,1\}^n$, define the string-level map
\[
\lambda^s(\bx_s) := \lambda(\bx)_s,
\]
interpreting $\bx_s$ as a string of $0$s and $1$s. The output $\lambda^s(\bx_s) \in \{0,1\}^m$ is also treated as a string. We use $\bx_s$ to distinguish string representation from vector notation.
\end{description}

\begin{Example}
If $\bx = (0,1,1,1)$, then $\operatorname{code}(\bx) = 7$ and $\bx_s = 0111$.
\end{Example}

\begin{Remark}[Binary strings with fixed length]
When it is helpful to emphasize that a binary string $\bx_s$ has a fixed number of bits, we use the subscript notation $\bx_{s[k]}$, where $k$ denotes the total length of the binary string. For example, both strings $0111$ and $00111$ represent the integer $7$, but $\bx_{s[4]} = 0111$ and $\bx_{s[5]} = 00111$ clarify the intended length. 
\end{Remark}

\vspace{2mm}
\noindent\textbf{Additional notation:}
\begin{description}
\item[Concatenation:] We use double vertical bars to denote concatenation: \\
if $\bx = 001$, then $0 \mathbin{\|} \bx := 0001$.
\item[Complement:] $\overline{\bx}$ is the complement of $\bx$, that is, $\mathbf{1} - \bx$; in the previous example, $\overline{\bx} = 110$.
\end{description}

From Lemma \ref{deflambda} for $\bx \in \{0,1\}^n$, we have
\[
\lambda^s(\bx_s) = \ind(x_1 = x_2)\, \ind(x_1 = x_3)\, \ldots\, \ind(x_{n-1} = x_n),
\]
and
\[
\lambda^\#(\operatorname{code}(\bx)) = \operatorname{code}\left( \ind(x_1 = x_2), \ind(x_1 = x_3), \ldots, \ind(x_{n-1} = x_n) \right),
\]
which is also the decimal interpretation of $\lambda^s(\bx_s)$.

For example, $\lambda^s(1000) = 000111$, $\lambda^s(1001) = 001100$, \ldots, $\lambda^s(1111) = 111111$. Also, $\lambda^\#(8) = 7$, $\lambda^\#(9) = 12$, \ldots, $\lambda^\#(15) = 63$. All $\lambda$-values for $n=4$ (and $n=3$) are shown in Table~\ref{table:n34}.

\begin{table}[h!]
\centering
\begin{tabular}{|c|c|c|c|}
\hline
$\bx_s$ & $\operatorname{code}(\bx)$ & $\lambda^s(\bx_s)$ & $\lambda^\#(\operatorname{code}(\bx))$ \\
\hline
100 & 4 & 001 & 1 \\
101 & 5 & 010 & 2 \\
110 & 6 & 100 & 4 \\
111 & 7 & 111 & 7 \\
\hline
\end{tabular}
\hspace{1cm}
\begin{tabular}{|c|c|c|}
\hline
$\bx_s$ & $\lambda^s(\bx_s)$ & $\lambda^\#(\operatorname{code}(\bx))$ \\
\hline
1000 & 000111 & 7 \\
1001 & 001100 & 12 \\
1010 & 010010 & 18 \\
1011 & 011001 & 25 \\
1100 & 100001 & 33 \\
1101 & 101010 & 42 \\
1110 & 110100 & 52 \\
1111 & 111111 & 63 \\
\hline
\end{tabular}
\caption{Values of $\lambda^s$ and $\lambda^\#$ for binary strings of length $n=3$ and $n=4$ starting with $1$.}
\label{table:n34}
\end{table}

These tables illustrate the mapping from binary strings to their encoded agreement probability vectors, providing concrete examples that will clarify the structure of our enumeration formulas.

One of the first things to notice in the tables is the monotonicity of $\lambdahash$. We will prove this important property at the end of this section. Prior to that, there is a lemma stating a recursive rule for the string-level map.

\begin{Lemma}[Recursive Rule for $\lambda^s$] \label{lemma:lambda-recursion}
Let $\bx_s \in \{0,1\}^k$, with $k \geq 2$. Then
\begin{align*}
\lambda^s(0 \mathbin{\|} \bx_s) &= \overline{\bx_s} \mathbin{\|} \lambda^s(\bx_s), \\
\lambda^s(1 \mathbin{\|} \bx_s) &= \bx_s \mathbin{\|} \lambda^s(\bx_s),
\end{align*}
where $\overline{\bx_s}$ denotes the bitwise complement of $\bx_s$.
\end{Lemma}

Before proving the lemma, we present two examples that illustrate the recursive rule.
\begin{Example}
Let $\bx_s = 100$. Then $\lambda^s(0 \mathbin{\|} \bx_s) = \lambda^s(0100)$. From Table~\ref{table:n34} ($n=4$), we know that
\[
\lambda^s(0100) = 011001.
\]
If we apply Lemma~\ref{lemma:lambda-recursion} to the same string,
\[
\lambda^s(0 \mathbin{\|} 100) = \overline{100} \mathbin{\|} \lambda^s(100) 
\]
and reading $\lambda^s(100)$ from Table~\ref{table:n34} ($n=3$), we obtain the same result $011 \mathbin{\|} 001 = 011001$.
\end{Example}

Let's take another example with a string that starts with $1$.
\begin{Example}
We see that $\lambda^s(1100) = 100001$ from Table~\ref{table:n34}. Applying Lemma~\ref{lemma:lambda-recursion}, the same string is obtained:
\[
\lambda^s(1 \mathbin{\|} 100) = 100 \mathbin{\|} \lambda^s(100) = 100 \mathbin{\|} 001 = 100001.
\]
\end{Example}

\begin{proof}
Let $\bx = (x_1, \ldots, x_k) \in \{0,1\}^k$. Then
\[
\lambda^s(0 \mathbin{\|} \bx) = \ind(0 = x_1)\, \ldots\, \ind(0 = x_k) \mathbin{\|} \lambda^s(\bx).
\]
Observe that
\[
\ind(0 = x_i) =
\begin{cases}
1 & \text{if } x_i = 0, \\
0 & \text{if } x_i = 1,
\end{cases}
= \overline{x_i},
\]
which proves the first identity. The second follows similarly, since
\[
\lambda^s(1 \mathbin{\|} \bx) = \ind(1 = x_1)\, \ldots\, \ind(1 = x_k) \mathbin{\|} \lambda^s(\bx),
\]
and clearly $\ind(1 = x_i) = x_i$.
\end{proof}

Throughout this paper, we use the standard lexicographic order on binary vectors, denoted $>_{\mathrm{lex}}$ and $\geq_{\mathrm{lex}}$, as defined below.

\begin{Definition}[Lexicographic Order] \label{def:lex}
Let $x = (x_1, x_2, \dots, x_n)$ and $y = (y_1, y_2, \dots, y_n)$ be two binary strings or vectors of length $n$, where each $x_i, y_i \in \{0,1\}$.  
We say that $x$ is less than or equal to $y$ in the \emph{lexicographic order}, denoted $x \leq_{\mathrm{lex}} y$, if one of the following holds:
\begin{itemize}
  \item $x = y$, or
  \item there exists an index $k \in \{1, \dots, n\}$ such that
  \begin{itemize}
    \item $x_i = y_i$ for all $i < k$, and
    \item $x_k < y_k$.
  \end{itemize}
\end{itemize}
In the case where $x \neq y$, we write $x <_{\mathrm{lex}} y$. Similarly, we define $x \geq_{\mathrm{lex}} y$ if $y \leq_{\mathrm{lex}} x$.
\end{Definition}

For example, $(0,1,1,0) <_{\mathrm{lex}} (0,1,1,1)$ and $(1,0,0) >_{\mathrm{lex}} (0,1,1)$. Also $0110 <_{\mathrm{lex}} 0111$ and $ 011 <_{\mathrm{lex}} 100$.

\begin{Remark} 
For binary strings of the same fixed length,
\[ \bx_s <_{\mathrm{lex}} \by_s \Leftrightarrow \code(\bx) < \code(\by).
\] 
\end{Remark}

\begin{Proposition}[Monotonicity of $\lambda^\#$] \label{prop:monotone}
Let $\bx, \by \in \{0,1\}^n$ be binary vectors with $x_1 = y_1 = 1$. If $\operatorname{code}(\bx) < \operatorname{code}(\by)$, then $\lambda^\#(\bx) < \lambda^\#(\by)$.
\end{Proposition}

\begin{proof}
Suppose $\bx_s$ and $\by_s$ are two binary strings of length $n$ with leading digit $1$, such that $\operatorname{code}(\bx) < \operatorname{code}(\by)$ and $\operatorname{code}(\bx) = 2^{n-1} + z$ for some $z < 2^{n-1}-1$, so that $\operatorname{code}(\by) = \operatorname{code}(\bx) + 1 = 2^{n-1} + z + 1$.

Then $\bx_s = 1 \mathbin{\|} z_s$ and $\by_s = 1 \mathbin{\|} (z+1)_s$, where $z_s$ and $(z+1)_s$ denote the $n-1$-bit binary representations of $z$ and $z+1$, respectively.

By Lemma~\ref{lemma:lambda-recursion}, we have:
\begin{align*}
\lambda^s(\by_s) &= (z+1)_s \mathbin{\|} \lambda^s((z+1)_s) \\
                 &>_{\mathrm{lex}} z_s \mathbin{\|} 11\ldots1 \\
                 &\geq_{\mathrm{lex}} z_s \mathbin{\|} \lambda^s(z_s) = \lambda^s(\bx_s).
\end{align*}
The first inequality follows from the strict increase in the left part of the string, and the second from the fact that $\lambda^s(z_s) \leq_{\mathrm{lex}} 11\ldots1$ where both sides have $\binom{n-1}{2}$ bits.
Clearly the inequality holds for integer-valued map also  and $\lambda^\#(\by) > \lambda^\#(\bx)$, as desired.
\end{proof}


\section{Encoded vertices of $\bone-\operatorname{CUT}(n)$}

In this section, we study the encoded vertices of the $\bone-\operatorname{CUT}(n)$ polytope using the $\lambda$-maps introduced previously.

\vskip2mm
\noindent\textbf{Vertices:} For every $n$, the polytope $\bone$-$\operatorname{CUT}(n)$ has $2^{n-1}$ vertices:
\[
\operatorname{vert}(\bone\text{-}\operatorname{CUT}(n)) = \{\nu^n(1),\ldots, \nu^n(2^{n-1})\}.
\]
Their encoded values for $k = 1,\ldots, 2^{n-1}$ are denoted by
\[
v^n(k) := \code(\nu^n(k)).
\]

\begin{Remark}\label{remark:v12} For $n=1$ and $2$ Remark \ref{remark:n12} can be directly extended to the vertices of $\bone- \cutn$ so 
\begin{itemize}
\item  $v^1=\{0\}$
\item $v^2=\{0,1\}$ .
\end{itemize}
\end{Remark}
Let $a_k^n := 2^{n-1} + k - 1$ for $k=1,\ldots,2^{n-1}$. These correspond to binary vectors of length $n$ with the first coordinate $1$, ordered lexicographically. For example, for $n=4$,  $a^4_1=8$ ($1000$), $a^4_2=9$ ($1001$), ..., $a^4_8=15$ ($1111$).

From Proposition~\ref{prop:lambda-cut-relation} and Remark~\ref{remark:lambda} now it stems the following Corollary.
\vskip6mm
\begin{Corollary}[Vertex Encoding]
Encoded vertices of $\bone-\operatorname{CUT}(n)$ are obtained by applying  the $\lambdahash$ operation to $a^n_k$:
\[
v^n(k) = \lambda^\#(a^n_k) = \lambda^\#(2^{n-1} + k - 1).
\]
\end{Corollary}

\vskip6mm
\noindent Now we study the values of $\lambda^\#$ in order to identify precisely all vertices $v^n(k)$.

\noindent Out of $2^{n-1}$ vertices, we can immediately identify two:
\[
v^n(1) = \lambda^\#(a^n_1) = 2^{\binom{n-1}{2}} - 1, \quad v^n(2^{n-1}) = \lambda^\#(a^n_{2^{n-1}}) = 2^{\binom{n}{2}} - 1.
\]

\noindent These boundary vertices represent the minimum and maximum encoded values, respectively.

The monotonicity established in Proposition~\ref{prop:monotone} ensures that vertex codes preserve the ordering of their indices:
\begin{Corollary}
The sequence $v^n(1),\ldots, v^n(2^{n-1})$ is increasing.
\end{Corollary}

\vskip4mm
\noindent
We now present an easily implementable algorithm that generates the encoded vertices $v^n(\cdot)$ for any $n$.

\subsection*{Algorithm 1: Generation of encoded vertices for $\bone-\cutn$} \label{alg:1}
Let $n \geq 2$ and $m = \binom{n}{2}$. This procedure generates all encoded vertices of the $\bone$-$\operatorname{CUT}(n)$ polytope.
\begin{enumerate}
  \item Enumerate all binary vectors $a_k = (a_k(1), a_k(2), \dots, a_k(n)) \in \{0,1\}^n$ such that $a_k(1)=1$, ordered by the integer value of their binary representation.
  \item For each vector $a_k$, $k=1,\ldots, 2^{n-1}$, define the associated encoded vector
  \[
  \lambda(k) = (\lambda(k)_{ij})_{1 \le i < j \le n} \in \{0,1\}^m
  \]
  by
  \[
  \lambda(k)_{ij} = \ind(a_k(i) = a_k(j)),
  \]
  where $\ind(\cdot)$ denotes the indicator function.
  \item Vertices of $\bone-\operatorname{CUT}(n)$ are then $\nu^n(k) := \lambda(k)$ and their integer codes values are
  \[
  v^n(k) := \code(\lambda(k)), \quad k=1,\ldots, 2^{n-1}.
  \]
\end{enumerate}
\vskip8mm
\begin{Example} \label{example456}
Applying the above algorithm for several values of $n$:
\begin{itemize}
\item $n=4$: $v^4 = \{7, 12, 18, 25, 33, 42, 52, 63\}$
\item $n=5$: $v^5 = \{63, 116, 170, 225, 281, 338, 396, 455, 519, 588, 658, 729, 801, 874, 948, 1023\}$
\item $n=6$: $v^6 = \{1023, 1972, 2922, 3873, 4825, 5778, 6732, 7687, 8647, 9612, 10578, 11545, 12513, \\
13482, 14452, 15423, 16447, 17524, 18602, 19681, 20761, 21842, 22924, 24007, 25095, 26188, \\
27282, 28377, 29473, 30570, 31668, 32767\}$
\end{itemize}
\end{Example}
\vskip6mm
When we plot the encoded vertices versus their indices, we notice a striking linear tendency. In Figure~\ref{fig:n7print}, encoded vertices $v^n(k)$ are plotted against $k$ for $n=7$.

\begin{figure}[h!]
\centering
\includegraphics[width=0.8\textwidth]{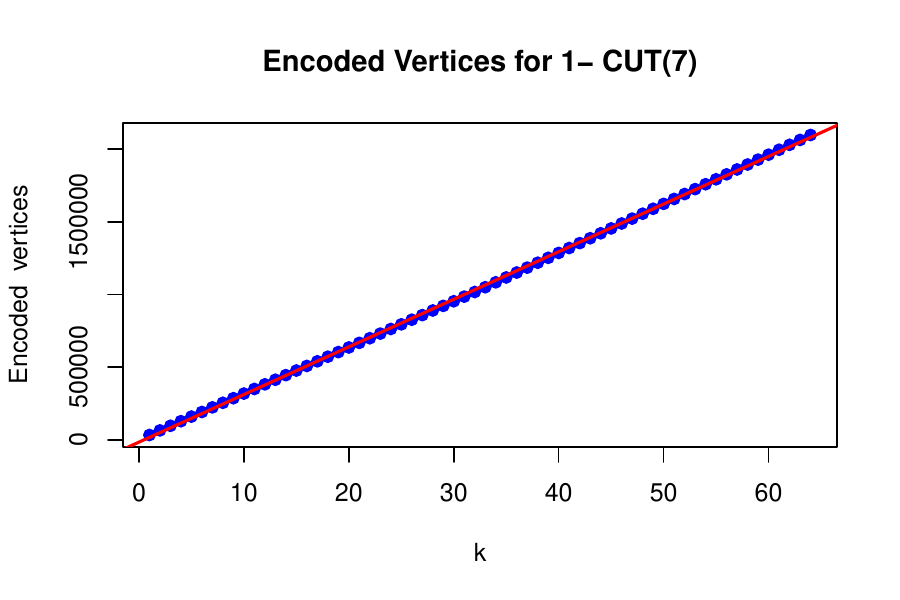}
\caption{Encoded vertices $v^n(k)$ plotted versus $k=1,\ldots, 2^{n-1}$, for $n=7$.}
\label{fig:n7print}
\end{figure}

Linear behaviour is preserved when $v^n$ are scaled by $2^{\binom{n-1}{2}}$. However, the vertices are not exactly on a line, so we refer to this as an \emph{almost linear ordering}. In Figure~\ref{fig:scaled}, we see that a line fitted to scaled $v^n$ is $y=x-1/2$. The residuals also follow a pattern and we investigate that below. 

\begin{figure}[h!] 
\centering
\includegraphics[height=0.45\textheight]{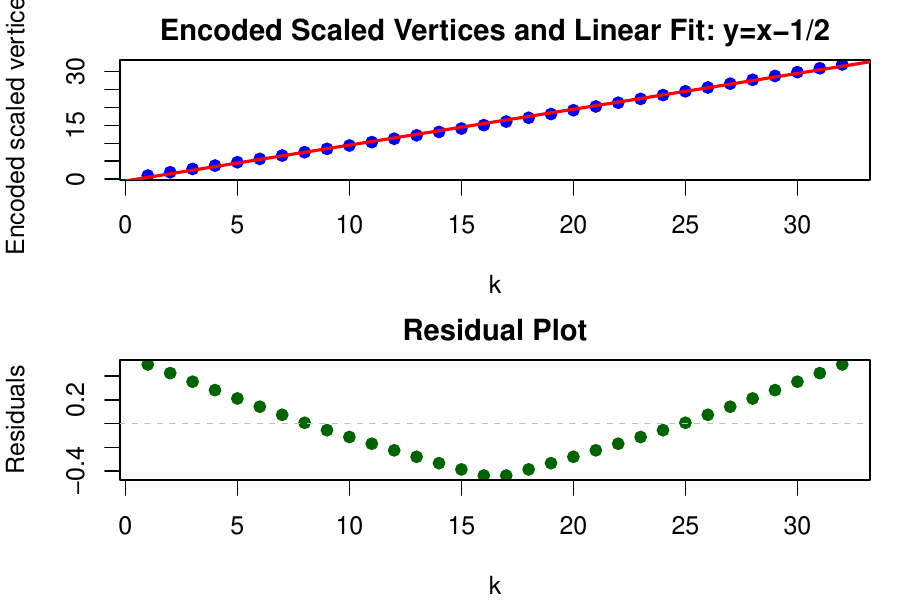}
\caption{Top: Scaled vertices $v^n(k)/2^{\binom{n-1}{2}}$ plotted versus $k=1,\ldots, 2^{n-1}$, for $n=6$. The fitted line $y= x - 0.5$ is obtained using linear regression. Bottom: Residuals from the linear fit, indicating small but structured deviations from linearity.}
\label{fig:scaled}
\end{figure}


The observed near-linear alignment of encoded vertices is not merely a numerical coincidence, but can be formalized as follows.


\begin{Proposition}[Almost linear ordering of vertices]
For $n \geq 2$, $k=1,\ldots,2^{n-1}$ and $v^n(k) = \lambda^\#(a_k^n)$,
\[
k-1 \leq \frac{v^n(k)}{2^{\binom{n-1}{2}}} < k.
\]
\end{Proposition}
\begin{proof}
Using the recursive rule for $\lambdas$ Lemma~\ref{lemma:lambda-recursion}, the string version of $v^n(k)$ satisfies
\[
v^n(k)_s = \lambda^s(1 \parallel (k-1)_{s[n-1]}) = (k-1)_{s[n-1]} \parallel \lambda^s((k-1)_{s[n-1]}).
\]
Therefore,
\begin{equation}\label{eq:firstrec}
v^n(k) = (k-1) 2^{\binom{n-1}{2}} + \code(\lambda^s((k-1)_{s[n-1]})).
\end{equation}
Since $\lambda^s((k-1)_{s[n-1]})$ has $\binom{n-1}{2}$ bits, $\code(\lambda^s((k-1)_{s[n-1]})) < 2^{\binom{n-1}{2}}$. Dividing by $2^{\binom{n-1}{2}}$ gives the result.
\end{proof}

Equations in the previous proof imply a recursive relation between $v^n(\cdot)$ and $v^{n-1}(\cdot)$ which we formalize next.

This recursive structure is further explored in Section~6, where we introduce the alternating cycle function and its role in the explicit formula for $v^n(k)$.

\subsubsection*{Recursion}

\begin{Proposition}[Recursive relation between $\vn$s] If  $n \geq 2$ and $\vn()$ are encoded vertices of $\bone-\cut(n)$ then 

\begin{eqnarray} \label{recursion}
\vn(k)=   2^{ n -1 \choose 2} (k-1) + 
 \begin{cases}
      v^{n-1}(2^{n-2}+1-k) & \text{  for  } k \in \{1,\ldots,2^{n-2}\}\\
       v^{n-1}(k-2^{n-2}) & \text{  for  } k \in \{2^{n-2}+1,\ldots,2^{n-1}\},
    \end{cases}    
\end{eqnarray}
 where $\binom{1}{2}=0$.
 \end{Proposition}
 
 \begin{proof}
 For $n \geq 2$, $k=1,\ldots,2^{n-1}$ let's look again at $\lambdas((k-1)_{s[n-1]})$ in Equation (\ref{eq:firstrec}). 
 \begin{enumerate}
 \item 
 If $k \in \{1,\ldots,2^{n-2}\}$ then $k-1 \in \{0,\ldots,2^{n-2}-1\}$ and in its binary represenatation using $n-1$ bits means that the first bit will be $0$. Then 
 \[
 \lambdas((k-1)_{s[n-1]})= \lambdas(0 \parallel (k-1)_{s[n-2]}) = \overline{(k-1)_{s[n-2]}} \parallel \lambdas((k-1)_{s[n-2]})
 \]
by Lemma \ref{lemma:lambda-recursion}.
Then, remembering that $\lambdas(\bx_{s[n-2]})=\lambdas(\overline{\bx_{s[n-2]}})$ and that $\code(\bar{\bx})=(2^{n-1}-1)-\code(\bx)$
\begin{eqnarray*}
\code( \lambdas((k-1)_{s[n-1]})) &=&   (2^{n-1}-1-(k-1))\,2^{n-2 \choose 2} + \code (\lambdas((k-1)_{s[n-2]}))  \nonumber \\
                                                    &=& ((2^{n-1}- (k-1))-1))\,2^{n-2 \choose 2} + \code (\lambdas((2^{n-1}- (k-1))_{s[n-2]})) \nonumber \\
            \mbox{by (\ref{eq:firstrec})}  &=& v^{n-1} (2^{n-2}-(k-1)). 
\end{eqnarray*}

\item For $k \in \{2^{n-2}+1,\ldots,2^{n-1}\}$ the binary represenatation of $k-1$ using $n-1$ bits and its complement in binary representation, whose integer value is $2^{n-1}-1-(k-1) \leq 2^{n-2}$ and we can apply  the previous case result
\[
 \code(\lambdas((k-1)_{s[n-1]}))= \code(\lambdas(2^{n-1}-k)_{s[n-1]})= v^{n-1}(2^{n-2}-2^{n-1}+k)=v^{n-1}(k-2^{n-2}).
 \]
\end{enumerate} 
Combining both cases yields the stated recursion.
 \end{proof}

The recursion  \ref{recursion} relates $k$ to $2^{n-2}+1-k$ or $k-2^{n-2}$ depending on whether $k$ is in the first or second half of the $\{1,\ldots,2^{n-1}\}$.  This map will be defined using a function that we named {\em Alternating Cycle Function}. 


Beyond their almost linear and recursive properties, the encoded vertices exhibit other properties governed by this integer-valued function introduced below.
\section{Alternating Cycle Function}

\begin{Definition}[Alternating Cycle Function] \label{def:SN}
For integers \(N \geq 2\), the \emph{Alternating Cycle function} \(S^N \colon \mathbb{N} \to \mathbb{N}\) is defined as:
\[
S^N(k) = 
\begin{cases} 
N + 1 - \left(k - \left\lfloor \frac{k-1}{N} \right\rfloor N\right), & \text{if } \left\lfloor \frac{k-1}{N} \right\rfloor \text{ is even}, \\ 
k - \left\lfloor \frac{k-1}{N} \right\rfloor N, & \text{if } \left\lfloor \frac{k-1}{N} \right\rfloor \text{ is odd}
\end{cases}
\]
for all $k \in \N$.
\end{Definition}

Having this definition we can rewrite the recursion between vertices $\vn$ as
\begin{equation} \label{eq:vertsn}
v^n(k) = 2^{\binom{n-1}{2}} (k - 1) + v^{n-1}\left(S^{2^{n-2}}(k)\right).
\end{equation}
This motivates study of the alternating cycle function which follows.

\vskip5mm
Throughout this work, \(N\) will always be a power of \(2\). Table~\ref{table:1} displays values of \(S^N(k)\) for \(N = 2, 4, 8, 16, 32\) and \(k = 1, \ldots, 16\). Figure~\ref{figure:S} illustrates the function's behavior for \(N = 2, 4, 8\) over \(k = 1, \ldots, 30\).

\begin{table}[h!]
\centering
\resizebox{\textwidth}{!}{%
\begin{tabular}{|c|c|c|c|c|c|c|c|c|c|c|c|c|c|c|c|c|} 
\hline
\(k\) & 1 & 2 & 3 & 4 & 5 & 6 & 7 & 8 & 9 & 10 & 11 & 12 & 13 & 14 & 15 & 16 \\ 
\hline
\(S^2(k)\) & 2 & 1 & 1 & 2 & 2 & 1 & 1 & 2 & 2 & 1 & 1 & 2 & 2 & 1 & 1 & 2 \\ 
\hline
\(S^4(k)\) & 4 & 3 & 2 & 1 & 1 & 2 & 3 & 4 & 4 & 3 & 2 & 1 & 1 & 2 & 3 & 4 \\ 
\hline
\(S^8(k)\) & 8 & 7 & 6 & 5 & 4 & 3 & 2 & 1 & 1 & 2 & 3 & 4 & 5 & 6 & 7 & 8 \\ 
\hline
\(S^{16}(k)\) & 16 & 15 & 14 & 13 & 12 & 11 & 10 & 9 & 8 & 7 & 6 & 5 & 4 & 3 & 2 & 1 \\ 
\hline
\(S^{32}(k)\) & 32 & 31 & 30 & 29 & 28 & 27 & 26 & 25 & 24 & 23 & 22 & 21 & 20 & 19 & 18 & 17 \\ 
\hline
\end{tabular}%
}
\caption{Values of the alternating cycle function \(S^N(k)\) for \(N = 2, 4, 8, 16, 32\) and \(k = 1, \ldots, 16\)}
\label{table:1}
\end{table}

\begin{Remark}[Key Properties of \(S^N\)] \label{remark:sn}
The alternating cycle function exhibits three fundamental properties:
\begin{enumerate}
\item \textbf{Periodicity}: \(S^N\) has a period of \(2N\), i.e., \(S^N(k + 2N) = S^N(k)\) for all \(k \in \mathbb{N}\).
\item \textbf{Wave Pattern}: Over each interval of length \(2N\), \(S^N(k)\) decreases from \(N\) to \(1\) and then increases from \(1\) to \(N\).
\item \textbf{Local Palindromy}: Within each \(2N\)-length block, the sequence is palindromic (symmetric under reversal).
\end{enumerate}
\end{Remark}

\begin{figure}[h!]
\centering
\includegraphics[width=0.8\textwidth]{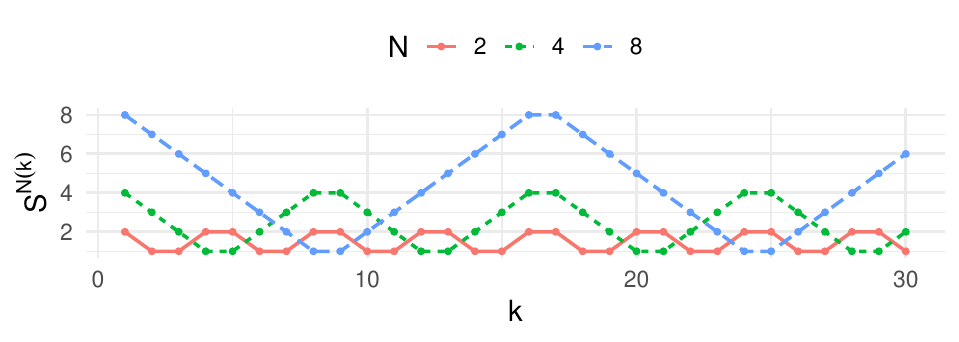}
\caption{Graphs of \(S^N(k)\) for \(N = 2, 4, 8\) and \(k = 1, \ldots, 30\).}
\label{figure:S}
\end{figure}

\begin{Remark}
\noindent
The alternating pattern of \( S^N \) resembles zigzag and discrete triangular wave sequences found in other fields such as signal processing. Here the function \( S^N \) is defined over the integers with a blockwise, parity-dependent structure. It exhibits composition properties that are not typically formalized in existing literature. For small values of \( N \), the resulting integer sequences do not appear in the OEIS [The On-line Encyclopedia of Integer Sequences], suggesting that this formulation and its structural properties have not been previously studied in this particular form.
\end{Remark}
\vskip5mm

\subsection*{Power-of-Two Compositions}
The alternating-cycle function exhibits composition invariance across powers of two in the period, characterized by the following identity:

\begin{Proposition}[Power-of-Two Composition Invariance] \label{Prop:fractal}
For \(N \geq 2\) and \(k \in \mathbb{N}\):
\begin{enumerate}
\item[i)] 
\[
S^N(S^{2N}(k)) = S^N(k).
\]
\item[ii)]  For all \(m \geq 1\),
\[
S^N(S^{2^m N}(k)) = S^N(k).
\]
\end{enumerate}
\end{Proposition}

\begin{proof}
\textbf{Part (i):} Since \(S^{2N}\) has period \(4N\), which is a multiple of the period of \(S^N\), it suffices to verify the identity for \(k \leq 4N\). By direct computation:
\[
S^{2N}(k) = 
\begin{cases} 
2N - k + 1, & \text{if } k \leq 2N, \\ 
k - 2N, & \text{if } k > 2N.
\end{cases}
\]
For \(k \leq 2N\):
\[
S^N(2N - k + 1) = 
\begin{cases} 
k - N, & \text{if } k \geq N + 1, \\ 
N - k + 1, & \text{if } k < N + 1
\end{cases}
= S^N(k).
\]
For \(2N < k \leq 4N\), periodicity gives \(S^N(k - 2N) = S^N(k)\).

\textbf{Part (ii):} We proceed by induction on \(m\). The base case (\(m = 1\)) is established in part (i). Assume the result holds for some \(m \geq 1\).
Taking $ S^{2^{m+1}N}(k)$ as a function argument in the hypothesis we obtain
 \begin{equation}\label{eq:m+1}
 S^N(S^{{2^m}N}(S^{2^{m+1}N}(k)))= S^N(S^{2^{m+1}N}(k))).
 \end{equation}
Now consider:
\[
S^N(S^{2^m N}(S^{2^{m+1} N}(k))) = S^N(S^{2^m N}(k)) = S^N(k),
\]
where the first equality follows from part (i) applied to \(2^m N\), and the second equality is the induction hypothesis. The generalized invariance follows.
\end{proof}

\subsection*{Self-Composition Properties}
The function \(S^N\) also exhibits notable regularity under composition with itself:

\begin{Proposition}[Self-Composition Properties] \label{prop:composition}
For \(S^N\) with \(N \geq 2\) and any \(k \in \mathbb{N}\):
\begin{enumerate}
\item[a)] \emph{Complementary Action}:
\[
S^N(S^N(k)) = N + 1 - S^N(k).
\]
\item[b)] \emph{Triple Composition Idempotence}:
\[
S^N \circ S^N \circ S^N = S^N.
\]
\end{enumerate}
\end{Proposition}

\begin{proof}
\textbf{Part (a):} Since \(S^N(k) \in \{1, \ldots, N\}\) by definition, we have:
\[
S^N(S^N(k)) = 
\begin{cases} 
N + 1 - S^N(k), & \text{if } S^N(k) \leq N, 
\end{cases}
= N + 1 - S^N(k).
\]

\textbf{Part (b):} Using part (a) twice:
\[
S^N(S^N(S^N(k))) = S^N(N + 1 - S^N(k)) = N + 1 - (N + 1 - S^N(k)) = S^N(k).
\]
\end{proof}

\section{Explicit Formula for \(v^n(\cdot)\)}
We are now ready to state our main result: an explicit, closed-form formula for the encoded vertices of $\bone$-$\operatorname{CUT}(n)$, valid for all $n$.
 We also include a worked example to illustrate the computation and conclude with implications for the original \(\operatorname{CUT}(n)\) polytope.

\begin{Theorem}[Formula for Encoded Vertices of \(\bone\)-\(\operatorname{CUT}(n)\)] \label{Thm:1}
Let \(n \in \mathbb{N}\) and \(k = 1, \ldots, 2^{n-1}\). The encoded vertices are given by:
\begin{itemize}
\item \(v^1(1) = 0\)
\item \(v^2(1) = 0\), \(v^2(2) = 1\)
\item For \(n \geq 3\):
\[
v^n(k) = 2^{\binom{n-1}{2}} (k - 1) + \sum_{j=1}^{n-2} 2^{\binom{j}{2}} \left( S^{2^j}(k) - 1 \right),
\]
where \(\binom{1}{2} = 0\).
\end{itemize}
\end{Theorem}

\begin{proof}
The proof proceeds by induction on \(n\).

\textbf{Base Cases:}
\begin{itemize}
\item For \(n = 1\) and \(n = 2\) it was observed (see Remark \ref{remark:v12}) that easily \(v^1(1) = 0\) and  \(v^2(1) = 0\), \(v^2(2) = 1\).
\item For \(n = 3\): Explicit calculation using \(S^2 = (2,1,1,2)\) yields:
\[
v^3 = (1, 2, 4, 7),
\]
matching the expected vertices.
\end{itemize}

\textbf{Inductive step:}  
Assume the formula holds for \(n-1\), that is:
\[
v^{n-1}(k) = 2^{\binom{n-2}{2}}(k - 1) + \sum_{j = 1}^{n - 3} 2^{\binom{j}{2}}(S^{2^j}(k) - 1)
\quad \text{for all } k \in \{1, \ldots, 2^{n-2}\}.
\]

Let us now consider \(v^n(k)\) for any \(k \in \{1, \ldots, 2^{n-1}\}\).  
From the recurrence (Equation~\eqref{eq:vertsn}):
\[
v^n(k) = 2^{\binom{n - 1}{2}}(k - 1) + v^{n - 1}(S^{2^{n-2}}(k)).
\]

Note that \(S^{2^{n-2}}(k) \in \{1, \ldots, 2^{n-2}\}\), regardless of whether \(k \leq 2^{n-2}\) or not.  
Hence, by the inductive hypothesis:
\[
v^{n - 1}(S^{2^{n-2}}(k)) = 2^{\binom{n - 2}{2}}(S^{2^{n-2}}(k) - 1) + \sum_{j = 1}^{n - 3} 2^{\binom{j}{2}}(S^{2^j}(S^{2^{n-2}}(k)) - 1).
\]

Now, by Proposition~\ref{Prop:fractal}, we have:
\[
S^{2^j}(S^{2^{n-2}}(k)) = S^{2^j}(k), \quad \text{for all } j \leq n-3.
\]

So:
\begin{align*}
v^n(k)
&= 2^{\binom{n-1}{2}}(k - 1) + 2^{\binom{n - 2}{2}}(S^{2^{n-2}}(k) - 1)
+ \sum_{j = 1}^{n - 3} 2^{\binom{j}{2}}(S^{2^j}(k) - 1) \\
&= 2^{\binom{n-1}{2}}(k - 1) + \sum_{j = 1}^{n - 2} 2^{\binom{j}{2}}(S^{2^j}(k) - 1),
\end{align*}
which completes the proof.
\end{proof}

To demonstrate the effectiveness and simplicity of the formula, we provide explicit examples for small values of~$n$.

\subsection*{Worked Example: Vertices of \(\bone\)-\(\operatorname{CUT}(5)\)}
For \(n = 5\), the formula simplifies to:
\[
v^5(k) = 64(k - 1) + \underbrace{(S^2(k) - 1)}_{\text{weight } 2^0} + \underbrace{2(S^4(k) - 1)}_{\text{weight } 2^1} + \underbrace{8(S^8(k) - 1)}_{\text{weight } 2^3}.
\]
Sample computations:
\begin{itemize}
\item \(k = 1\): \(v^5(1) = 0 + (2 - 1) + 2(4 - 1) + 8(8 - 1) = 63\)
\item \(k = 4\): \(v^5(4) = 64 \cdot 3 + (2 - 1) + 2(1 - 1) + 8(5 - 1) = 225\)
\item \(k = 16\): \(v^5(16) = 64 \cdot 15 + (2 - 1) + 2(4 - 1) + 8(8 - 1) = 1023\)
\end{itemize}
The complete vertex set is:
\[
v^5 = (63, 116, 170, 225, 281, 338, 396, 455, 519, 588, 658, 729, 801, 874, 948, 1023),
\]
validating Algorithm~\ref{alg:1}.

\vspace{1em}
We now extend this result to the original \(\operatorname{CUT}(n)\) polytope.

To this end let's name the vertices of the cut polytope as $\omega^n :=\operatorname{vert}(\cutn)$ and $w^n=\code(\omega^n)$. Then
$w^1(1)=0$ and $w^2(1)=1, w^2(2)=0$. 

\begin{Corollary}[Encoded Vertices of \(\operatorname{CUT}(n)\)]
For \(n \geq 3\), the encoded vertices \(w^n(k)\) of \(\operatorname{CUT}(n)\) are given by:
\[
w^n(k) = 2^{\binom{n}{2}} - 1 - \left[ 2^{\binom{n-1}{2}} (k - 1) + \sum_{j=1}^{n-2} 2^{\binom{j}{2}} \left( S^{2^j}(k) - 1 \right) \right].
\]
\end{Corollary}
\begin{proof}
This follows directly from the relation $\omega^n(k) = \bone - \nu^n(k)$ and the encoding map \(\code(\cdot)\) definition.
\end{proof}


\section{Concluding Remarks}

\noindent
This work presents a complete description of the vertices of the $\bone$-$\operatorname{CUT}(n)$ ($\cutn$) polytope, valid for all $n$. . 
This result stands out as a rare case of an explicit, closed-form vertex enumeration formula for a central 0/1-polytope - such formulas are not known for prominent cases like the stable set or matching polytopes, where only algorithmic or implicit combinatorial listings are available.

The construction is grounded in a probabilistic interpretation and realized through a binary encoding that reveals a recursive and nearly linear structure in the vertex values. The vertex-generating function that emerges, called the alternating cycle function, exhibits a recursive zigzag behavior with a key power-of-two composition invariance (Proposition~\ref{Prop:fractal}) that drives the enumeration formula. It also displays interesting compositional properties that may have broader applicability.

Beyond potential applications in combinatorial optimization involving $\cutn$, the explicit vertex set derived here may serve as a benchmark for testing or validating vertex enumeration algorithms, particularly those targeting 0/1-polytopes.

The methods developed in this paper originate from a specific connection between multivariate symmetric Bernoulli distributions and the $\cutn$ polytope. Whether similar ideas can be extended to other polytope families - such as those associated with non-symmetric Bernoulli variables - remains an open and promising direction for future investigation.


\section*{Appendix: Vertex Generation Algorithm}
\label{sec:algorithm}

The explicit vertex formula yields a straightforward generation procedure:

\begin{enumerate}
    \item \textbf{Input:} Integer $n \geq 2$.
    \item \textbf{Initialization:} 
    \begin{itemize}
        \item Compute $N = 2^{n-1}$ (number of vertices)
        \item Compute $d = \binom{n}{2}$ (dimension of binary vectors)
    \end{itemize}
    \item \textbf{Define auxiliary function:} The alternating cycle function $S^M(k)$ for integer $M$ and $k \in \{1,\ldots,2M\}$:
    \[
    S^M(k) = 
    \begin{cases}
    M + 1 - k, & \text{if } 1 \leq k \leq M \\
    k - M, & \text{if } M+1 \leq k \leq 2M
    \end{cases}
    \]
    \item \textbf{Vertex generation:} For each $k = 1, \ldots, N$:
    \begin{enumerate}
        \item Compute vertex code:
        \[
        v^n(k) = 2^{\binom{n-1}{2}}(k-1) + \sum_{j=1}^{n-2} 2^{\binom{j}{2}} \left(S^{2^j}(k) - 1 \right)
        \]
        \item Convert $v^n(k)$ to $d$-bit binary representation to obtain the vertex of $\bone$-$\operatorname{CUT}(n)$
    \end{enumerate}
    \item \textbf{Output:} Complete set of $2^{n-1}$ vertices of $\bone$-$\operatorname{CUT}(n)$.
\end{enumerate}

\begin{Remark}
The algorithm generates all vertices in a single pass without iterative refinement or specialized data structures.
\end{Remark}


\bibliographystyle{plain}
\bibliography{Berncut}

\end{document}